\documentclass[a4paper]{amsart}
\usepackage{amsmath,amssymb,amsfonts}
\usepackage{mathrsfs,latexsym,amsthm,enumerate}
\usepackage{amscd}
\usepackage[all]{xy}

\renewcommand{\d}{\mathbf{d}}
\renewcommand{\r}{\mathbf{r}}

\newtheorem{theorem}{Theorem}[section]
\newtheorem{lemma}[theorem]{Lemma}
\newtheorem{corollary}[theorem]{Corollary}

\newtheorem{proposition}[theorem]{Proposition}

\title{The structure of generalized inverse semigroups}

\dedicatory{Dedicated to the memory of Professor John Mackintosh Howie (1936--2011)}

\author{Ganna Kudryavtseva}
\address{Ganna  Kudryavtseva, University of Ljubljana,
Faculty of Computer and Information Science, 
Tr\v{z}a\v{s}ka cesta 25,
SI-1001, Ljubljana,
 SLOVENIA}
\email{ganna.kudryavtseva\symbol{64}fri.uni-lj.si}
\author{Mark V. Lawson}
\address{Mark V.Lawson, Department of Mathematics
and the
  Maxwell Institute for Mathematical Sciences, 
Heriot-Watt University,
Riccarton,
Edinburgh~EH14~4AS,
SCOTLAND}
\email{markl@ma.hw.ac.uk}

\thanks{The first author was partially supported by an ARRS grant P1-0288, and the second by an EPSRC grant EP/I033203/1}
 
\begin{document}

\begin{abstract}
We prove that the structure of right generalized inverse semigroups is determined by free \'etale actions of inverse semigroups.
This leads to a tensor product interpretation of Yamada's classical struture theorem for generalized inverse semigroups.
\end{abstract}

\maketitle
\section{Introduction}

A {\em generalized inverse semigroup} is a regular semigroup whose set of idempotents forms a normal band.
Such semigroups play an unexpected role in inverse semigroup theory.
This was first observed in \cite{L1}, where they were used in constructing enlargements of inverse monoids,
and most recently in \cite{AL} where this construction was generalized 
to yield a procedure for manufacturing all inverse semigroups Morita equivalent to a given inverse semigroup.

Our interest in this class of semigroups arose from our efforts to find a common generalization of our papers
\cite{K} and \cite{L3,L5,LL} where we each obtained a different non-commutative generalization of classical Stone duality.
We plan to describe the results of this collaboration in later papers, whereas in this one we shall focus on some new results
on the structure of generalized inverse semigroups that arose from this work.

It may seem surprising that there is anything new to say about them.
That there is, is due to comparatively recent developments within inverse semigroup theory, 
in particular the theory of \'etale actions described in Section~3, 
combined with an interpretation of presheaves of sets over meet semilattices that seems to go back to the work of Wagner \cite{W} 
and is described in Section~2.
For the remainder of this section, we shall summarize some of the key results about generalized inverse semigroups
that we shall need.

Throughout this paper we shall call upon results on regular semigroups and we refer the reader to the standard reference \cite{H2}
for all undefined terms and unproved results concerning such semigroups.

An element $s$ of a semigroup $S$ is said to be {\em (von Neumann) regular} if there is an element $t$,
called an {\em inverse} of $s$, such that $s = sts$ and $t = tst$.
The set of inverses of the element $s$ is denoted by $V(s)$.
In an inverse semigroup $S$, we write the unique inverse of $s$ as $s^{-1}$ and we define
$\mathbf{d}(s) = s^{-1}s$ and $\mathbf{r}(s) = ss^{-1}$. 
On an orthodox semigroup $S$, the relation $\gamma$ defined by
$$s \, \gamma \, t \Leftrightarrow V(s) \cap V(t) \neq \emptyset \Leftrightarrow V(s) = V(t)$$ 
is the minimum inverse congruence.
As usual, we denote Green's relations on any semigroup by $\mathscr{L},\mathscr{R},\mathscr{H},\mathscr{D}$ and $\mathscr{J}$.
A regular semigroup is said to be {\em $\mathscr{R}$-unipotent} if each $\mathscr{R}$-class contains a unique idempotent.
The class of {\em $\mathscr{L}$-unipotent} semigroups is defined dually.
Venkatesan \cite{V} proved that $\mathscr{R}$-unipotent semigroups are orthodox with a 
left regular band of idempotents.
On a regular semigroup $S$, we may define a relation $a \leq b$ if and only if $a = eb = bf$ for some idempotents $e$ and $f$.
This is a partial order called the {\em natural partial order}.
If $S$ is a band, the order is the usual order on idempotents: $e \leq f$ if and only if $e = ef = fe$.
If $S$ is a semigroup and $e \in S$ an idempotent, the subsemigroup $eSe$ is called a {\em local submonoid}.
A semigroup is said to have a property {\em locally} if each local submonoid has it.
The properties of the natural partial order are a convenient way of organizing some of the classes of semigroups
studied in this paper. The two-sided results below are due to Nambooripad and may be found in \cite{H2}, 
and the one-sided results are due to Blyth and Gomes \cite{BG}

\begin{proposition} \mbox{}
\begin{enumerate}

\item The natural partial order is right compatible with the multiplication if and only if
the semigroup is locally $\mathscr{L}$-unipotent.

\item A band is $\mathscr{L}$-unipotent if and only if it is right regular, that is it satisfies the identity $efe = fe$, 
if and only if the $\mathscr{L}$-relation is equality.
On such a semigroup $\gamma = \mathscr{R}$.

\item The natural partial order is left compatible with the multiplication if and only if
the semigroup is locally $\mathscr{R}$-unipotent.

\item A band is $\mathscr{R}$-unipotent if and only if it is left regular, that is it satisfies the identity $efe = ef$, 
if and only if the $\mathscr{R}$-relation is equality.
On such a semigroup $\gamma = \mathscr{L}$.

\item The natural partial order is compatible with the multiplication 
if and only if the semigroup is locally inverse.

\item A band is locally inverse if and only if it is normal, that is it satisfies the identity $efgh = egfh$.

\item A band is right normal, that is it satisfies the identity $efg = feg$, if and only if it is normal and $\mathscr{L}$-unipotent.

\item A band is left normal, that is it satisfies the identity $efg = egf$, if and only if it is normal and $\mathscr{R}$-unipotent.

\end{enumerate}
\end{proposition}

Generalized inverse semigroups were introduced by Yamada \cite{Y1} back in 1967.
They then seem to have become subsumed within the general theory of orthodox semigroups \cite{Hall,Y2,Y3,Y4}.
In particular, rather than orthodox semigroups with a normal band of idempotents
those with a regular band of idempotents were the main focus of attention
and their relationship with $\mathscr{L}$- and $\mathscr{R}$-unipotent semigroups.

In this paper, we really do need normal bands rather than the more general regular bands
and so we shall describe the theory given in \cite{Y3} restricted to this case.
To do so, we shall need two special classes of generalized inverse semigroups.
A {\em left generalized inverse semigroup} is a regular semigroup whose set of idempotents is a left normal band
and a {\em right generalized inverse semigroup} is a regular semigroup whose set of idempotents is a right normal band.
Yamada obtained two, related, structure theorems for generalized inverse semigroups.
In \cite{Y1}, he showed that generalized inverse semigroups could be described in terms of inverse semigroups and left and right normal bands,
whereas in \cite{Y3} the theory, restricted as indicated above,
shows how to describe them in terms of a left generalized inverse semigroup and a right generalized inverse semigroup.
The two theorems are closely related but it is the second that is of most interest to us now and we shall return to the first in Section~5.

Let $S$ be a generalized inverse semigroup.
Following \cite{Y3}, we define the relation $\lambda$ on $S$ 
by $\lambda = \gamma \cap \mathscr{L}$.
The relation $\rho$ is defined dually. 
  
\begin{proposition} Let $S$ be a generalized inverse semigroup.
Then $\lambda$ (respectively, $\rho$) is the minimum right (respectively, left) generalized inverse congruence on $S$.
It restricts to the $\mathscr{L}$-relation (respectively, $\mathscr{R}$-relation) on the set of idempotents
and is idempotent pure.
The intersection of $\lambda$ and $\rho$ is equality.
\end{proposition}
\begin{proof} We prove first that $\lambda$ is a congruence.
It is clearly a right congruence and so it only remains to prove that it is a left congruence.
Suppose that $a \, \lambda \, b$ and $c \in S$ is arbitrary.
Then $a \, \gamma \, b$ and $a \mathscr{L} b$.
Clearly, $ca \, \gamma \, cb$.
It therefore only remains to prove that $ca \, \mathscr{L} \, cb$.
Let $a' \in V(a)$, $b' \in V(b)$ and $c' \in V(c)$.
Since  $a \mathscr{L} b$ we have that $ab'b = a$ and $ba'a = b$.
Now observe that
$$(cba'c')ca = cba'c'ca = c (ba')(c'c) a = c(c'c)(ba')a = cb$$ 
using the fact that $ba'$ is an idempotent and that we are working in a generalized inverse semigroup.
We have therefore proved that 
$$cb = (cba'c')ca.$$ 
We may similarly show that $ca = (cab'c')cb$.
It follows that $ca \mathscr{L} cb$, as required.

We next show that $S/\lambda$ is a right generalized inverse semigroup.
Denote the $\lambda$-congruence class containing $s$ by $[s]$.
Let $e,f \in E(S)$.
Then $[e] \mathscr{L} [f]$ implies that $[ef] = [e]$ and $[fe] = [f]$.
It follows that $e = efe$, $f = fef$, $ef \mathscr{L} e$ and $fe \mathscr{L} f$.
It now readily follows that $e \mathscr{L} f$ and so $[e] = [f]$, as required. 

It only remains to prove that $\lambda$ is the {\em minimum} right generalized inverse congruence on $S$.
Let $\sigma$ be any right generalized inverse congruence on $S$ and $a\, \lambda \, b$.
Let $a' \in V(a)$ and $b' \in V(b)$.
Since $\lambda$ is a congruence, it follows that $aa'\, \lambda \, ba'$ and $a'a \, \lambda \, a'b$.
But $aa'\, \mathscr{L} \, ba'$ and $a'a \, \mathscr{L} \, a'b$
and both $ba'$ and $a'b$ are idempotents. 
It follows that $\sigma (aa') = \sigma (ba')$ and $\sigma (a'a) = \sigma (a'b)$.
Hence $\sigma (a) = \sigma (b) \sigma (a'a)$ and $\sigma (a) = \sigma (aa') \sigma (b)$.
This implies that $\sigma (a) \leq \sigma (b)$.
The reverse inequality follows by symmetry.

The proof of the final assertion follows from the fact that on an orthodox semigroup $\gamma \cap \mathscr{H}$ is the equality relation.
\end{proof}

If $S$ is a generalized inverse semigroup then the map
$S \rightarrow S/\rho \times S/\lambda$ given by $s \mapsto (\rho (s),\lambda (s))$ is
an injective homomorphism 
since in any orthodox semigroup, we have that $\gamma \cap \mathscr{H}$ is just equality.
The result on generalized inverse semigroups below now follows by the above and Theorem~3 of \cite{Y3}.

\begin{theorem} A regular semigroup is a generalized inverse semigroup 
if and only if it is a subdirect product of a left generalized inverse semigroup
by a right generalized inverse semigroup.
\end{theorem}

For right generalized inverse semigroups there is even an analogue of symmetric inverse monoids due to Madhavan \cite{Mad}
which we briefly describe.
 In this description, we write functions to the right of their arguments.
Let $X$ be a non-empty set and $\rho$ an equivalence relation defined on $X$.
The set $M_{\rho}(X)$ consists of all the partial functions $\alpha$ of $X$ that satisfy three additional conditions with respect
to the equivalence relation $\rho$.
First, if $x \, \rho \, y$ then $(x)\alpha = (y)\alpha$;
second, if $(x)\alpha  \, \rho \, (y)\alpha$ then $x \, \rho \, y$;
and third, if $(x)\alpha$ is defined and $x \, \rho \, y$ then $(y)\alpha$ is defined.
The semigroup $M_{\rho}(X)$ is called the {\em symmetric right generalized inverse semigroup}
and every right generalized inverse semigroup can be embedded in such a semigroup.
Observe that in this semigroup, the idempotents are the partial functions $\alpha$ such that $x \, \rho \, (x)\alpha$ for all 
$x \in \mbox{dom}(\alpha)$.
It follows that if $\rho$ is chosen to be the equality relation then $M_{\rho}(X)$ is just the symmetric inverse monoid on $X$.\\

\noindent
{\bf Encomium }John Howie's two books \cite{H1,H2} are an indispensible reference for anyone wanting to learn the fundamentals of semigroup theory
and for all researchers in regular semigroups.
The second author first encountered semigroup theory through \cite{H1}, and part of his mathematical apprenticeship was attending the
semigroup meetings organized on a regular basis by John at St.~Andrews.

 \section{Presheaves of sets from a semigroup perspective}
 
The key idea that lies behind the work of this paper can be traced back to Wagner \cite{W}, and forms the subject of this section.
We begin with a definition.
Let $E$ be a meet semilattice equipped with the following additional data.
For each $e \in E$, let $X_{e}$ be a set
where we assume that if $e \neq f$ then $X_{e}$ and $X_{f}$ are disjoint.
If $e \geq f$ then a function $|_{f}^{e} \colon X_{e} \rightarrow X_{f}$ is given where $x \mapsto x|_{f}^{e}$.
We call these {\em restriction functions}.
In addition, $|_{e}^{e}$ is the identity on $X_{e}$ and if $e \geq f \geq g$ then 
$$(x|_{f}^{e})|_{g}^{f} = x|_{g}^{e}.$$
Put $X = \bigcup_{e \in E} X_{e}$ and define $p \colon X \rightarrow E$ by $p(x) = e$ if $x \in X_{e}$.
We shall say that $X = (X,p)$ is a {\em presheaf of sets over $E$}.
We will sometimes denote this presheaf by $X \stackrel{p}{\rightarrow} E$.
Observe that we do not assume that the sets $X_{e}$ are non-empty.
If they are all non-empty we denote the presheaf of sets by 
$X \stackrel{p}{\twoheadrightarrow} E$ and say that the presheaf has {\em global support.}

Presheaves of sets play a fundamental role, of course, in topos theory
and are usually viewed from an order-theoretic perspective.
Within semigroup theory, they are the basis of `strong semilattices of structures'.
We shall now describe a third way of thinking about them which forms the basis of this paper.

Our main idea is that presheaves of sets over semilattices can also be viewed as purely {\em algebraic} structures in the following way.
Let  $X \stackrel{p}{\rightarrow} E$ be a presheaf of sets.
Define a binary operation $\circ $ on $X$ as follows
$$x \circ y = y \vert^{p(y)}_{p(x) \wedge p(y)}.$$
It is routine to check that $(X,\circ)$ is a right normal band.
Thus from each presheaf of sets over a meet semilattice --- an order-theoretic structure --- 
we can construct a right normal band --- an algebraic structure.
In the case where the presheaf has global support, we can also go in the opposite direction.

\begin{theorem}\label{th: presheaves_right_normal_bands}
The category of presheaves of sets with global support is equivalent to the category of right normal bands.
\end{theorem}
\begin{proof} The category of right normal bands, $\mathcal{NB}$, has as objects the right normal bands and as morphisms
semigroup homomorphisms.
The category of presheaves $\mathcal{PS}$ has as objects presheaves of sets over meet semilattices with global support
$X \stackrel{p}{\twoheadrightarrow} E$.
If $(X,p)$ and $(Y,q)$ are two objects then a morphism from $(X,p)$ to $(Y,q)$
is a pair of functions $(\alpha,\beta)$ where $\alpha \colon X \rightarrow Y$ is a function,
$\beta \colon E \rightarrow F$ is a semilattice homomorphism, $q \alpha = \beta p$
and $\alpha (x \vert^{e}_{f}) = \alpha (x)\vert^{\beta (e)}_{\beta (f)}$ when $f \leq e$ and $p(x) = e$.

We define a functor from $\mathcal{NB}$ to $\mathcal{PS}$. 
Let $B$ be a right normal band and put $E = B/\mathscr{R}$, a semilattice since $\mathscr{R}$ is the minimum
semilattice congruence on $B$.
Define $p \colon B \rightarrow E$ by $p(e) = [e]$, the $\mathscr{R}$-class containing $e$.
This is clearly surjective.
If $[f] \leq [e]$ define the map from $R_{[e]}$ to $R_{[f]}$ to be $x \mapsto fx$. 
Then this defines on $B$ the structure of a presheaf of sets over $E$ such that
the multiplication induced by the presheaf structure coincides with the original multiplication on $B$.
We prove that the definition of the structure mapping is independent of the choice of idempotent $f$.
Suppose that $f \mathcal{R} f'$.
Then $f' = ff'$ and $f = f'f$.
We need to prove that $fx = f'x$.
But $fx = f'fx = ff'x = f'x$, using the fact that $B$ is a right normal band.
The proof of the remaining claims is now straightforward.
Now let $\theta \colon B_{1} \rightarrow B_{2}$ be a homomorphism between two right normal bands.
Such a homomorphism preserves the $\mathscr{R}$-relation.
It follows that we may define $\theta' \colon B_{1}/\mathscr{R} \rightarrow B_{2}/\mathscr{R}$
by $\theta'([b_{1}]) = [\theta (b_{1})]$.  
Thus $(\theta,\theta')$ is a morphism from $(B_{1},p_{1})$ to $(B_{2},p_{2})$.

We now define a functor from $\mathcal{PS}$ to $\mathcal{NB}$.
Let $X \stackrel{p}{\twoheadrightarrow} E$ be a presheaf with global support.
We have already seen how to construct a right normal band $X^{\circ} = (X,\circ)$.
Let $(\alpha,\beta) \colon (X,p) \rightarrow (Y,q)$ be a morphism of presheaves.
We shall define a homomorphism $\theta \colon X^{\circ} \rightarrow Y^{\circ}$
by $\theta (x) = \alpha (x)$.
It is routine to check that this is a semigroup homomorphism.

We now iterate the two constructions in order to check that we have an equivalence of categories.
Let $B$ be a right normal band whose multiplication is denoted by concatenation.
We have that
$$x \circ y = y\vert^{p(y)}_{p(x)p(y)} = y\vert^{[y]}_{[x][y]} = y \vert^{[y]}_{[xy]} = (xy)y = xy,$$
as required.
It's clear that a map between two right normal bands is also returned to itself under iteration of
the two functors.

Let $(X,p)$ be a presheaf of sets.
Observe that in the semigroup $X^{\circ}$ we have that $x \, \mathscr{R} \, y$
if and only if $p(x) = p(y)$.
It follows that the presheaf constructed from $X^{\circ}$ is isomorphic to $(X,p)$.
It's clear that a map between two presheaves is returned to an isomorphic copy 
under iteration of the two functors.
\end{proof}

Thus by using right normal bands, the notion of a presheaf of sets is made algebraic.
Our theorem is a variant of the classical structure theorem for normal bands 
which describes them as strong semilattices of right zero semigroups.
The connection between them is made using the observation that every non-empty set 
$A$ may be turned into a right zero semigroup by defining $ab = b$ for all $a,b \in A$.
However, the change of perspective represented by our theorem is important in this paper.

We now define two relations on presheaves of sets over semilattices and then explore some of their properties.
Define the relation $\leq$ on $X$ by $x \leq y$ if and only if $x = y\vert^{p(y)}_{p(x)}$.
This is a partial order.
Define $x \sim y$, and say that $x$ and $y$ are {\em compatible}, if and only if $\exists x \wedge y$ and $p(x \wedge y) = p(x) \wedge p(y)$.

\begin{lemma}\label{le: order_compatibility} Let $X \stackrel{p}{\rightarrow} E$ be a presheaf of sets.
\begin{enumerate}

\item If $x,y \leq z$ then $x \sim y$.

\item If $x \sim y$ and $p(x) \leq p(y)$ then $x \leq y$.

\item  $x \leq y$ if and only if $x = x \circ y$.

\item $x \sim y$ if and only if $x \circ y = y \circ x$.

\end{enumerate}
\end{lemma}
\begin{proof}
(1). By definition $x = z \vert^{p(z)}_{p(x)}$ and $y = z \vert^{p(z)}_{p(y)}$.
Put $u = z \vert^{p(z)}_{p(x) \wedge p(y)}$.
Then it is easy to check that $u = x \wedge y$ and $p(u) = p(x) \wedge p(y)$.

(2). By assumption, we have that $p(x \wedge y) = p(x)$.
It follows that $x = x \wedge y$ and so $x \leq y$.

(3). Let $x \leq y$.
Then $x = y \vert^{p(y)}_{p(x)}$.
It follows that $x = x \circ y$.
Conversely, let $x = x \circ y$.
Then $x = y \vert^{p(y)}_{p(x) \wedge p(y)}$
and so $x \leq y$.

(4). Suppose that $x \sim y$.
By definition $x \circ y = y \vert^{p(y)}_{p(x) \wedge p(y)}$.
Observe that $x \circ y \leq y$, $x \wedge y \leq y$ and
$p(x \circ y) = p(x \wedge y)$.
It follows that $x \circ y = x \wedge y$.
By symmetry we get that $y \circ x = x \wedge y$.
It follows that $x \circ y = y \circ x$.

Conversely, suppose that $x \circ y = y \circ x$.
Observe that $x \circ y \leq y$ and $y \circ x \leq x$.
It follows that $z = x \circ y = y \circ x \leq x,y$.
Let $u \leq x,y$.
Then $u = u \circ x$ and $u = u \circ y$.
It follows that $u = u \circ (x \circ y)$ and so $u \leq x \circ y$.
We have proved that $x \wedge y = x \circ y = y \circ x$.
It is immediate that $p(x \wedge y) = p(x) \wedge p(y)$.
\end{proof}

The following will be applied in Section~5.

\begin{corollary}\label{cor: order}
Let $X \stackrel{p}{\rightarrow} E$ be a presheaf of sets.
If $x,y \leq z$ and $p(x) \leq p(y)$ then $x \leq y$.
\end{corollary}

\section{\'{E}tale actions}

One of the central developments in inverse semigroup theory in recent years has been the recognition of the important role played
in the theory by what are called \'etale actions of inverse semigroups.
These were first explicitly defined in \cite{FS} and applied to the study of the Morita theory of inverse semigroups in \cite{S}.
But their origin lies in the cohomology theory of inverse semigroups \cite{Lausch}.
Proposition~\ref{prop: etale_actions} below goes back to \cite{Log} and \cite{LS}.

Let $S$ be an inverse semigroup and $X$ a non-empty set. 
A {\em left $S$-action} of $S$ on $X$ is a function $S\times X\to X$, defined by $(s,x)\mapsto s\cdot x$ (or $sx$),
such that $(st)x=s(tx)$ for all $s,t\in S$ and $x\in X$. {\em Right actions} are defined dually.  
If $S$ acts on $X$ we say that $X$ is an {\em $S$-set}.
In this paper, all actions will be assumed left actions unless stated otherwise.
A {\em left \'{e}tale action} $(S,X,p)$ of $S$ on $X$ is defined as follows \cite{FS,S}. 
Let $E(S)$ denote the semilattice of idempotents of $S$. There is a function $p:X\to E(S)$ and a left action $S\times X\to X$ such that the following two conditions hold:
\begin{enumerate}[(E1)]
\item $p(x) \cdot x=x$;
\item $p(s \cdot x)=sp(x)s^{-1}$.
\end{enumerate}
If $p$, above, is surjective we say that the action has {\em global support.}
A {\em morphism} $\varphi \colon (S,X,p)\to (S,Y,q)$ of left \'{e}tale  actions is a map $\varphi \colon X\to Y$ such that $q(\varphi(x))=p(x)$ 
for any $x\in X$ and $\varphi(s\cdot x)=s\cdot \varphi(x)$ for any $s\in S$ and $x\in X$.

We begin this section by explaining how \'etale actions of inverse semigroups are related to actions of
inverse semigroups on presheaves of sets.

Let $S$ be an inverse semigroup with semilattice of idempotents $E(S)$ and let $X \stackrel{p}{\rightarrow} E(S)$
be a presheaf of sets over $E(S)$.
Denote by $S\ast X$ the set of all pairs $(s,x)\in S\times X$ such that $\d(s)=p(x)$. 
We say that there is a {\em left action of $S$ on the presheaf $X$} if for each pair $(s,x)\in S\ast X$ 
there is a unique element $s\cdot x\in X$ such that the following axioms hold:
\begin{enumerate}[({AP}1)]
\item\label{ap1}If $(e,x)\in S\ast X$ with $e\in E(S)$ then $e\cdot x=x$.
\item\label{ap2} If $(s,x)\in S\ast X$ then $p(s\cdot x)= \mathbf{r}(s)$.
\item\label{ap3} Suppose $\d(s)=\r(t)$. Then $(s, t\cdot x)\in S\ast X$ if and only if $(st, x)\in S\ast X$
and in the case when $(s, t\cdot x)\in S\ast X$ we have $s\cdot (t\cdot x)=(st)\cdot x$.
\item\label{ap4}  Let $f\leq \d(s)$. Then the following diagram commutes:
$$
\xymatrix{
X_{\r(s)}\ar[d]^{\vert_{\r(sf)}^{\r(s)}}&X_{\d(s)}\ar[l]_{s\cdot}\ar[d]^{\vert_f^{\mathbf{d}(s)}}\\ 
X_{\r(sf)}&X_f\ar[l]_{sf\cdot}
}
$$
That is we have $(s\cdot x)\vert_{\r(sf)}^{\r(s)}=sf\cdot (x\vert_f^{\d (s)})$ for any $(s,x)\in S\ast X$ and any $f\leq \d(s)$.
\end{enumerate}

Let $(X,p)$ and $(Y,q)$ be presheaves of sets over $E(S)$ and actions of $S$ on these presheaves are given. 
A {\em morphism $\varphi:(X,p)\to (Y,q)$} of left actions of $S$ is defined as a map $\varphi:X\to Y$ such that
\begin{enumerate}
\item $\varphi(X_e)\subseteq Y_e$ for any $e\in E(S)$;
\item If $(s,x)\in S\ast X$ then $s\cdot \varphi(x)=\varphi(s\cdot x)$;
\item If $e\geq f$ and $x\in X_e$ then $\varphi(x)|_f^{e} =\varphi(x|_f^{e})$.
\end{enumerate}

Our first result is an observation that will be important.

\begin{proposition}\label{prop: presheaf_etale} A presheaf of sets $p \colon X \rightarrow E$ over a meet semilattice $E$ is the same thing
as a left \'etale action $E \times X \rightarrow X$ with respect to $p$.
\end{proposition}
\begin{proof} Let $(X,p)$ be an \'etale act with respect to $E$.
Define $x \circ y = p(x) \cdot y$.
It is easy to check that $(X,\circ)$ is a right normal band.
Observe also that $x \, \mathscr{R} \, y$ if and only if $p(x) = p(y)$.
Conversely, let $X \rightarrow E$ be a presheaf.
Define $e \cdot x = x \vert^{p(x)}_{ep(x)}$.
It is routine to check that we get an \'etale action.
\end{proof}

By the above result, it follows that with every \'etale action $(S,X,p)$ there is
an {\em underlying presheaf} namely $(E(S),X,p)$ where the action $E(S) \times X \rightarrow X$
is defined by restriction.

The following proposition connects left \'{e}tale actions of $S$ and left actions of $S$ on presheaves of sets. 
It can be considered as an action analogue of the Ehresmann-Schein-Nambooripad theorem \cite{L2}.
We don't claim any novelty for it since it seems to be part of the folklore of \'etale actions.

\begin{proposition}\label{prop: etale_actions} Let $S$ be an inverse semigroup.
The category of left \'{e}tale actions of $S$ (with global support) is isomorphic to the category of left actions of $S$ on presheaves 
of sets over $E(S)$ (with global support).
\end{proposition}
\begin{proof} We connect left \'{e}tale actions on $X$ and left actions on presheaves of sets as follows. 
Let $(S,X,p)$ be a left \'{e}tale action of $S$.
By Proposition~\ref{prop: presheaf_etale},
there is an underlying presheaf $(E(S),X,p)$
where $X_e=p^{-1}(e)$ and if $e\geq f$ then the restriction from $X_e$ to $X_f$ is given by $x\vert_f^{e} = f\cdot x$. 
Now the action of $S$ on the presheaf $X\stackrel{p}{\rightarrow} E(S)$ is the given \'{e}tale action: for $(s,x)\in S\ast X$ we put
$sx=s\cdot x$
where we use concatenation to denote the presheaf action.

Conversely, given an action of $S$ on a presheaf $X\stackrel{p}{\rightarrow} E(S)$ we define the action of $S$ on $X$ as follows:
$$
s\cdot x= (sp(x))(x|_{\d(s)p(x)}^{p(x)}).
$$
It can be verified that this is indeed a left \'{e}tale action,  such that $s\cdot x=sx$ if $(s,x)\in S\ast X$. 
Any morphism of left \'{e}tale actions can be regarded as a morphism of the actions of $S$ on presheaves of sets and vice versa. 
The statement now easily follows.
\end{proof}




The goal of the remainder of this section is to define what we mean by free \'etale actions
and to describe them explicitly.
In Proposition~\ref{prop: presheaf_etale}, we showed that with each \'etale action $(S,X,p)$ there is an induced presheaf $(E(S),X,p)$.
We now show that this functor from left \'etale $S$-sets to \'etale left $E(S)$-sets has a left adjoint.

Let $S$ be an inverse semigroup and let $(E(S),A,q)$ be a presheaf of sets 
where we use the identification established in Proposition~\ref{prop: presheaf_etale}.
Put
$$S \ast A = \{ (s,a) \in S \times A \colon \mathbf{d}(s) = q(a)\}$$
which is just a pullback.
Define $r \colon S \ast A \rightarrow E(S)$ by $r(s,a) = \mathbf{r}(s)$
and define $S \times (S \ast A) \rightarrow S \ast A$ by $s \cdot (t,a) = (st,\mathbf{d}(st) \cdot a)$.
This is well-defined because $\mathbf{d}(st) \leq \mathbf{d}(t)$.
 
\begin{proposition}\label{prop: free_etale_sets} 
With the above definitions $(S,S \ast A,r)$ is a left \'etale action
and we have constructed a left adjoint to the forgetful functor above.
\end{proposition}
\begin{proof} To verify that we have an action, we need to check that
$(st) \cdot (u,a) = s \cdot (t \cdot (u,a))$.
But this follows from the definition and the fact that $\mathbf{d}(stu) \leq \mathbf{d}(tu)$.

(E1) holds because $r(t,a) \cdot (t,a) = \mathbf{r}(t) \cdot (t,a) = (t, \mathbf{d}(t) \cdot a) = (t,a)$.

(E2) holds because $\mathbf{r}(st) = s \mathbf{r}(t)s^{-1}$.

Let $(E(S),A,q)$ be a presheaf.
There is a presheaf morphism $\alpha \colon A \rightarrow S \ast A$ given by $a \mapsto (q(a),a)$.
Let $\beta \colon A \rightarrow X$ be a presheaf morphism to the presheaf induced from $(S,X,p)$.
Define $\theta (s,a) = s \cdot \beta (a)$.
Then this is a morphism of \'etale sets and is unique such that $\theta \alpha = \beta$.
\end{proof}

We shall refer to \'etale actions of the form $(S,S \ast A,r)$ as {\em free \'etale sets}.

\section{The structure of right generalized inverse semigroups}

The goal of this section is to prove the following theorem.

\begin{theorem} Let $T$ be a fixed inverse semigroup.
The category of free left \'etale $T$-sets is equivalent to the category of all right generalized inverse 
semigroups $S$ for which there is a surjective homomorphism $\theta \colon S \rightarrow T$ where the kernel of $\theta$ is $\gamma$.
\end{theorem}

We show first how to construct a free \'etale set from a semigroup.

\begin{proposition} Let $S$ be a right generalized inverse semigroup such that $T = S/\gamma$.
Regard $S$ as a set and define $T \times S \rightarrow S$ by $[a] \cdot s =  as$.
In addition, define 
$p \colon S \rightarrow E(S/\gamma)$ by $p(s) = [ss']$ where $s' \in V(s)$. 
Then $S$ is a free \'etale set. 
\end{proposition}
\begin{proof}
We show first that the action is well-defined.
Let $a \, \gamma \, b$.
We prove that $as = bs$.
Let $a' \in V(a)$.
Then $a' \in V(b)$ since $a \, \gamma \, b$.
Thus both $a'a$ and $a'b$ are idempotents.
Observe that $a'a \, \gamma \, a'b$ and so $a'a \, \mathscr{R} \, a'b$ 
since we are working in a right normal band.
Thus $a'a = a'ba'a$.
Hence $a'as = a'ba'as = a'aa'bs = a'bs$ since the idempotents form a right normal band.
Let $b' \in V(b) = V(a)$ and observe that from $a \, \gamma \, b$ we get that $aa' \, \gamma \, bb'$
and so $aa'\, \mathscr{R} \, bb'$.
In particular, $aa'bb' = bb'$.
Thus $as = aa'bs = aa'bb'bs = bb'bs = bs$, as required. 
It is now immediate that $S/\gamma \times S \rightarrow S$ defines an action.

Next we show that the action is \'etale.
The map $p$ is well-defined since if $s'' \in V(s)$ then $ss' \, \mathscr{R} ss''$ and so $[ss'] = [ss'']$.
It is clearly surjective.
We prove that 
$(S/\gamma, S, p)$ is \'etale with global support.
(E1) holds: we have that $p(a) \cdot a = [aa'] \cdot a = aa'a = a$.
(E2) holds: $p([a] \cdot s) = p(as) =  [ass'a'] = [a]p(s)[a]^{-1}$.

It remains to show that the action is free.
To do this we need a presheaf over $E(S/\gamma)$.
The obvious candidate is $E(S)$ itself.
By restriction, we therefore have a presheaf $(E(S/\gamma),S,q)$ where $q(e) = [e]$ the $\mathscr{R}$-class of $e$.
Form the free \'etale set $S/\gamma \ast E(S)$.
Its elements are ordered pairs $([s],e)$ such that $s's \,\mathscr{R}\, e$.

We prove first that there is a bijection between $S$ and the set $S/\gamma \ast E(S)$.
We shall use the following definition.

Let $\theta \colon S \rightarrow T$ be a surjective homomorphism of regular semigroups.
We say that it is an {\em $\mathscr{L}$-cover} if for each idempotent $e \in S$ the map
$(\theta \mid L_{e}) \colon L_{e} \rightarrow L_{\theta (e)}$ is bijective.
We prove that the natural map $S \rightarrow S/\gamma$ is an $\mathscr{L}$-cover.

Suppose first that $s \, \mathscr{L} \, t$ and $\gamma (s) = \gamma (t)$.
Let $s' \in V(s)$.
By assumption, $s' \in V(t)$.
Thus $ts'$ is an idempotent.
Since $\mathscr{L}$ is a right congruence $ss' \, \mathscr{L} \, ts'$.
Thus $ss' = ss'ts'$.
It follows that $ss' = (ss')(ts')(ss')$.
By assumption, $E(S)$ is a right normal band.
Thus $ss' = (ts')(ss')(ss')$.
That is $s = ts's$.
But $s \mathscr{L} t$
and $s's \mathscr{L} s$ so that $s's \mathscr{L} t$.
It follows that $ts's = t$ and so $s = t$, as required.

Next, let $e \in E(S)$ and $\gamma (t) \, \mathscr{L} \, \gamma (e)$. 
Let $t' \in V(t)$.
Then $\gamma (t't) \, \mathscr{L} \, \gamma (e)$.
Therefore $\gamma (t't) = \gamma (e)$ since in an inverse semigroup $\mathscr{L}$-related idempotents are equal.
We now use the observation that in a band
$$i \, \gamma \, j \Leftrightarrow i = iji \mbox{ and } j = jij.$$
It follows that $e = et'te$ and $t't = t'tet't$.
Consider the element $te \in S$.
Then $\gamma (te) = \gamma (t) \gamma (e) = \gamma (t)$.
From $e = (et') te$ it is immediate that $te \mathcal{L} e$.

Define $\kappa \colon S \rightarrow S/\gamma \ast E(S)$ by $\kappa (s) = ([s], s's)$.
This is well-defined and a bijection by what we proved above.
It is routine to check that in this way we have defined an isomorphism of \'etale sets.
\end{proof}

We shall now show how to construct a semigroup from a free \'etale set.

\begin{proposition}\label{prop: yamada}
Let $T$ be an inverse semigroup and $(E(T),X,p)$ a presheaf of sets with global support.
Constuct the free \'etale set
$$S = T \ast X = \{ (t,x) \in T \times X \colon \d (t) = p(x) \}.$$
Define a binary operation on $S$ by
$$(s,x)(t,y) = s \cdot (t,y) = (st, \mathbf{d}(st) \cdot y ).$$
With the above binary operation, $S$ is a right generalized inverse semigroup and $S/\gamma$ is isomorphic to $T$.
\end{proposition}
\begin{proof} The proof of associativity is routine;
idempotents have the form $(p(x),x)$ and the multiplication of idempotents is isomorphic to
the right normal band structure on $X$;
the inverses of $(s,x)$ are all elements  of the form $(s^{-1},y)$;
and $(s,x) \, \gamma \, (t,y)$ if and only if $s = t$.
\end{proof}

It remains to show that the two functors we have defined form part of an equivalence of categories.
We define a {\em right Yamada semigroup} to be a semigroup constructed from a free \'etale set in accordance with
Proposition~\ref{prop: yamada}.

\begin{theorem}\label{th: structure_of_right_generalized_inverse_semigroups} 
Every right generalized inverse semigroup is isomorphic to a right Yamada semigroup.
\end{theorem}
\begin{proof} Let $S$ be a right generalized inverse semigroup.
We have already defined a bijection $\kappa \colon S \rightarrow S/\gamma \ast E(S)$.
It just remains to prove that it is a homomorphism.
We have that 
$\kappa (s) = ([s],s's)$, 
$\kappa (t) = ([t],t't)$ 
and
$\kappa (st) = ([st],(st)'st)$. 
It is straightforward to check that $\kappa (st) = \kappa (s)\kappa (t)$.
\end{proof}

Finally, we need to go in the other direction.
The proof of the following is also immediate.

\begin{proposition} 
Let $(T,T \ast X,p)$ be a free \'etale set.
Then the free \'etale set constructed from the associated right generalized inverse semigroup
is isomorphic to $(T,T \ast X,p)$.
\end{proposition}

\section{The structure of generalized inverse semigroups}

In this section, we use the theory developed in this paper together with some Morita theory to explain  
Yamada's main structure theorem for generalized inverse semigroups  \cite[Theorem 2]{Y1}.
We start with a description of his theorem in our terms.

Let $T$ be an inverse semigroup and let $(X,p)$ and $(Y,q)$ be two presheaves of sets over $E(T)$ with global support.
The former will be regarded as a left normal band and the latter as a right normal band.
Both will be handled using the approach of Proposition~3.1.
Given this data, put
$$\mathscr{Y}(X,T,Y)=\{(x,t,y)\in X \times T \times Y \colon p(x) = \mathbf{r}(t) \text{ and } q(y) = \mathbf{d}(t) \}.
$$
Define a binary operation on $\mathscr{Y}(X,T,Y)$ by
$$
(x,s,y)(u,t,v)=(x \cdot \mathbf{r}(st),st, \mathbf{d}(st) \cdot v ).
$$
Yamada's theorem is that $\mathscr{Y}(X,T,Y)$ is a generalized inverse semigroup and every generalized inverse semigroup is of this form.
We will refer to the semigroup $\mathscr{Y}(X,T,Y)$ as a {\em Yamada semigroup}.  

Let $(x,s,y) \in \mathscr{Y}(X,T,Y)$.
Then
$$V((x,t,y)) = \{(u,t^{-1},v) \colon p(u) = \mathbf{d}(t) \mbox{ and } q(v) = \mathbf{r}(t) \},$$
and is non-empty since both our presheaves are assumed to have global support.
It follows that 
$$(x,t,y) \, \gamma \, (u,s,v) \Leftrightarrow t = s.$$ 
The idempotents in $\mathscr{Y}(X,T,Y)$ are those elements of the form $(x,e,y)$ where $e$ is an idempotent.
A simple calculation shows that
$$(x,s,y) \mathscr{L} (u,t,v) \Leftrightarrow s \mathscr{L} t \mbox{ and } y = v.$$
A dual result holds for $\mathscr{R}$.
It follows that
$$(x,t,y) \, \lambda \, (u,s,v) \Leftrightarrow t = s \mbox{ and } y = v.$$ 
We may therefore identify the right generalized inverse semigroup
$\mathscr{Y}(X,T,Y)/\lambda$ with the set of pairs $(t,y)$ where $q(y) = \mathbf{d}(t)$,
and a multiplication given by $(s,y)(t,v) = (st, \mathbf{d}(v))$.
A dual result holds for $\rho$.

The Morita theory of regular semigroups is due to \cite{T1,T2} and was developed in \cite{L4}.
Let $S$ be an orthodox semigroup with minimum inverse congruence $\gamma$.
McAlister proved \cite[Proposition 1.4]{DM} that the natural homomorphism $S \rightarrow S/\gamma$ is a local isomorphism
if and only if $S$ is generalized inverse.
From the theory developed in \cite{L4}, it follows that a generalized inverse semigroup $S$ is Morita equivalent to the inverse semigroup $S/\gamma$.
We refer the reader to Talwar's paper \cite{T2} for the Morita theory we use here, as well as \cite{H2} for the theory of tensor products.
One definition from the theory of tensor products will be needed below.
Let $X$ be a right $S$-set and $Y$ be a left $S$-set.
Then a function $f \colon X \times Y \rightarrow Z$, to a set $Z$,
is said to be {\em balanced} if $f(x \cdot s,y) = f(x,s \cdot y)$ for all $x \in X$, $y \in Y$ and $s \in S$.

Let $S$ be a generalized inverse semigroup.
We construct right and left generalized inverse semigroups $S_{R} = S/\lambda$ and $S_{L} = S/\rho$ respectively.
By the theory developed in the previous section and its dual,
the semigroup $S_{R}$ is a free left \'etale $S/\gamma$-set and $S_{L}$ is a free right \'etale $S/\gamma$-set.
Regarding both as simply $S$-sets,
we may therefore form their tensor product $S_{L} \otimes S_{R}$.

We now describe a general construction.
Let $P$ be a left $R$-set and $Q$ a right $R$-set.
Suppose that there is an $(R,R)$-bilinear function $\langle,\rangle \colon P \times Q \rightarrow R$,
meaning that $\langle rp, q \rangle = r \langle p, q \rangle$ and $\langle p, qs \rangle = \langle p, q \rangle s$
for all $r,s \in R$.
Then the tensor product $Q \otimes P$ becomes a semigroup when we define
$$(q \otimes p)(q' \otimes p') = q \otimes \langle p, q' \rangle p'.$$

Just such a bilinear map can be defined in our case.
We write $[s]$ to mean the $\gamma$-class of $s$,
$[e]_{L}$ to mean the $\mathscr{L}$-class of the idempotent $e$ in the band of idempotents,
and 
$[e]_{R}$ to mean the $\mathscr{R}$-class of the idempotent $e$ in the band of idempotents.
Define $\langle, \rangle \colon S_{R} \times S_{L} \rightarrow S/\gamma$ by
$$\langle ([s], [e]_{L}), ([f]_{R}, [t]) \rangle = [st].$$
This is a bilinear map.
It follows that there is a semigroup product defined on $S_{L} \otimes S_{R}$.
The semigroup  $S_{L} \otimes S_{R}$ is Morita equivalent to the inverse semigroup $S/\gamma$.
It is called the {\em Morita semigroup over $S/\gamma$ defined by $\langle,\rangle$.}

\begin{theorem} 
Yamada semigroups are isomorphic to Morita semigroups. 
\end{theorem}
\begin{proof} Let $S$ be a generalized inverse semigroup.
We shall prove that $S_{L} \otimes S_{R}$ is isomorphic to the Yamada semigroup isomorphic to $S$.
In what follows, we may therefore assume that $S$ is given as a Yamada semigroup.
To simplify notation we shall write $xsy$ for a typical element of $\mathscr{Y} (X,T,Y)$.
The semigroups $S_{L}$ and $S_{R}$ have concrete descriptions in terms of this notation.
The semigroup $S_{L}$ has elements $xs$ and the semigroup $S_{R}$ has elements $ty$.

A typical element of  $S_{L} \otimes S_{R}$ therefore has the form $xs \otimes ty$.
Using the properties of tensor products, we have that
$$xs \otimes ty 
= x \r (s) \otimes s \cdot (ty) 
= x \r (s) \otimes st(\d (st) \cdot y)
= (x \cdot \r(st)) \r (st) \otimes st (\d (st) \cdot y)$$
It follows that each element of $S_{L} \otimes S_{R}$
can be written in the form 
$$x \r (s) \otimes sy$$
where $p(x) = \r(s)$ and $q(y) = \d (s)$.
We shall call such elements {\em normalized}.
We now calculate the product of two normalized elements
$$(x \r(s) \otimes sy)(u \r(t) \otimes tv).$$
This is equal to
$$x \r (s) \otimes s \r(t) \cdot (tv),$$
which normalizes to
$$(x \cdot \mathbf{r}(st)) \r (st) \otimes st (\mathbf{d}(st) \cdot v).$$ 
Define the function 
$\theta \colon \mathscr{Y}(X,T,Y) \rightarrow S_{L} \otimes S_{R}$
by 
$\theta (x,s,y) = x \r(s) \otimes sy$.
We have proved so far that this is a surjective homomorphism.
It remains to show that this homomorphism is injective.
This is the same as showing that two normalized elements of $S_{L} \otimes S_{R}$
are equal precisely when they `look equal'.
Observe that the map $S_{R} \times S_{L} \rightarrow S$
defined by $(xs,ty) \mapsto st$ is balanced.
It follows that if $x \r(s) \otimes sy = u \r(t) \otimes tv$ then $s = t$.
Hence $p(x) = p(u) = \r (s)$ and $q(y) = q(v) = \d (t)$.
Therefore, we suppose that we have two normalized elements such that
$x \r(s) \otimes sy = u \r(s) \otimes sv$.

In order to show that two tensors are equal, we apply a succession of {\em left moves} and {\em right moves}.
A left move has the form $(xa,by) \rightarrow (x_{1}a_{1},b_{1}y_{1})$.
This has the following properties:
$x_{1} \leq x$, $y \leq y_{1}$, $ab = a_{1}b_{1}$ and, of course, $p(x_{1}) = \mathbf{r}(a_{1})$
and $q(y_{1}) = \mathbf{d}(b_{1})$.
A right move has the same form and the same properties except that $x \leq x_{1}$, $y_{1}\leq y$.
We denote a finite sequence of left and right moves by $\stackrel{\ast}{\rightarrow}$.
Observe that we may assume that we begin with a left move and end with a right move
because there are trivial left and right moves that do not change the ordered pair.
We may also assume that left and right moves alternate.

Suppose therefore that $(x\mathbf{r}(s),sy) \stackrel{\ast}{\rightarrow} (u\mathbf{r}(s),sv)$.
There are elements $a_{1}, \ldots, a_{n}$ and $b_{1}, \ldots, b_{n}$ in $S$ such that
$s = a_{1}b_{1} = a_{2}b_{2} = \ldots = a_{n}b_{n} = s$
and elements $x_{1}, \ldots, x_{n} \in X$ and $y_{1}, \ldots, y_{n} \in Y$
such that $x \geq x_{1} \leq x_{2} \geq x_{3} \leq \ldots, \geq x_{n} \leq u$
and
$y \leq y_{1} \geq y_{2} \leq y_{3} \geq \ldots \leq y_{n} \geq v$.
In addition, $p(x_{i}) = \mathbf{r}(a_{i})$ and $q(y_{i}) = \mathbf{d}(b_{i})$.
By assumption, $p(x) = p(u)$ and $q(y) = q(v)$.
Our goal is to prove that $x = u$ and $y = v$.
Observe that $s = a_{i}b_{i}$.
It follows that $\mathbf{r}(s) \leq \mathbf{r}(a_{i})$
and $\mathbf{d}(s) \leq \mathbf{d}(b_{i})$.
Hence $p(x) \leq p(x_{i})$ and $q(y) \leq q(y_{i})$ for all $i$.
Given the symmetry of the situation, we need only prove explicitly that $x = u$.

The first step is easy.
Since $x_{1} \leq x$ and $p(x) \leq p(x_{1})$ we have that $x = x_{1}$.
Assume that $x \leq x_{1} \leq \ldots \leq x_{i}$.
We shall prove that $x \leq x_{i+1}$.
If $x = x_{i} \leq x_{i+1}$ already, then there is nothing to prove.
If, on the other hand, $x_{i+1} \leq x_{i} = x$ then we have $p(x) = p(x_{i+1})$ and so $x = x_{i+1}$.
It follows that $x \leq u$.
But by assumption, $p(x) = p(u)$ and so $x = u$, as required.
\end{proof}


\end{document}